\documentclass{amsart}
\usepackage{amssymb, amsmath, amsthm, graphics, comment, xspace, enumerate}
\usepackage{hyperref}
\usepackage{color}
\newtheorem{thm}{Theorem}[section]

\theoremstyle{definition}
\newtheorem{defn}[thm]{Definition}

\theoremstyle{remark}
\newtheorem{rem}[thm]{Remark}
\numberwithin{equation}{section}

\begin{document}
\title[Multidimensional Widder--Arendt theorem...]{Multidimensional Widder--Arendt theorem in locally convex spaces}

\author{Marko Kosti\' c}
\address{Faculty of Technical Sciences,
University of Novi Sad,
Trg D. Obradovi\' ca 6, 21125 Novi Sad, Serbia}
\email{marco.s@verat.net}

{\renewcommand{\thefootnote}{} \footnote{
\\ \text{  }  \ \    {\it Key words and phrases.} Multidimensional vector-valued Laplace transform, multidimensional Widder--Arendt theorem, sequentially complete locally convex spaces.}}

\begin{abstract}
In this research article, we formulate and prove multidimensional Widder--Arendt theorem and integrated form of multidimensional Widder--Arendt theorem for functions with values in sequentially complete locally convex spaces. Established results seem to be new even for scalar-valued functions. 
\end{abstract} 

\maketitle 

\section{Introduction and preliminaries}

The Widder representation theorem for Laplace transform was proved in 1934 (\cite{widder-t,widder}). Following the ideas of I. Miyadera (\cite{mija}), S. Zaidman proved in 1960 that the Widder representation theorem holds for functions with values in a reflexive Banach space $X$ (\cite{zaidman}). In 1987,
W. Arendt proved that the Widder representation theorem holds for functions with values in a Banach space $X$ if and only if $X$ has the Radon-Nikodym property, which was essentially important for a further expansion of the theory of abstract ill-posed Cauchy problems (\cite{arendt}).
In 1989, M. Hieber proved
an integrated version of the Widder theorem in arbitrary Banach space (\cite{hiber}). After that, the Widder--Arendt theorem has been seriously  reconsidered by A. Bobrowski \cite{bob}-\cite{bob1}, W. Chojnacki \cite{choj}-\cite{choj1}, J. 
Kisy\'nski \cite{kisa}, Y.-C. Li, S.-Y. Shaw \cite{li-shaw} and T.-J. Xiao, J. Liang \cite{x26}; cf. also \cite[Theorem 2.2.1, Theorem 2.4.1, Theorem 2.4.2]{a43}. The main aim of this paper is to reconsider the Widder--Arendt theorem
for multidimensional vector-valued Laplace transform (\cite{mvlt}). 

The structure and main ideas of paper can be briefly summarized as follows. We first explain the basic notation used throughout the paper and provide some useful observations about the Widder--Arendt theorem in the one-dimensional setting; cf. Theorem \ref{waxl} and Theorem \ref{waxl1} in Subsection \ref{oned}. We recall the basic definitions and results about multidimensional vector-valued Laplace transform in Subsection \ref{sdf}. The multidimensional Widder-Arendt theorem in locally convex spaces with Lipschitz-Radon-Nikodym property, which seems to be introduced here for the first time, is formulated and proved in Section \ref{lrnp}; cf. Theorem \ref{wa2}.
Motivated by results established in \cite{hiber}, \cite{li-shaw} and \cite{x26}, in Section \ref{kon} we prove 
integrated form of multidimensional Widder--Arendt theorem in arbitrary sequentially complete locally convex space; cf. Theorem \ref{maha}. We can freely say that the results established in this paper, like the famous Hille-Yosida theorem, are of pure theoretical interest in functional analysis.
\vspace{0.1cm}

\noindent {\bf Notation and preliminaries.} By  $X$ we denote a Hausdorff sequentially complete locally convex space\index{sequentially complete locally convex space!Hausdorff} over the field of complex numbers (SCLCS).
The abbreviation $\circledast$ stands for the fundamental system of seminorms\index{system of seminorms} which defines the topology of $X$; by $X^{\ast}$ we denote the dual space of $X$ equipped with the strong topology. For further information concerning the integration of functions with values in SCLCSs, we refer the reader to \cite{FKP}. If $p\in \circledast$, then we define $U_{p}:=\{x\in X : p(x)\leq 1\} $ and $U_{p}^{\circ}:=\{x^{\ast} \in X^{\ast} :  |\langle x^{\ast},x\rangle| \leq 1\mbox { for all }x\in U_{p}\}.$

The Gamma function\index{function!Gamma} is denoted by $\Gamma(\cdot)$ and the principal branch is always used to take the powers. Define $g_{\zeta}(t):=t^{\zeta-1}/\Gamma(\zeta)$ and
$0^{\zeta}:=0$ ($\zeta>0,$ $t>0$); $g_{0}(t):=\delta(t)$ is the Dirac distribution. Generally, the finite convolution product $\ast_{0}$ is defined by
$
(a\ast_{0}b)(t):=\int^{t}_{0}a(t-s)b(s)\, ds,$ $t\geq 0;$ $\delta \ast_{0} a\equiv a.$ If $0<T<+\infty,$ then we say that a function $f : [0,T] \rightarrow X$ is Lipschitz continuous if for each seminorm $p\in \circledast$ there exists a real constant $M_{p}>0$ such that $p(f(t)-f(t'))\leq M_{p}|t-t'|$ for all $t,\ t'\in [0,T];$ a function $f : [0,+\infty) \rightarrow X$ is said to be Lipschitz continuous if for each $T\in (0,+\infty)$ the function $f_{| [0,T]} : [0,T] \rightarrow X$ is Lipschitz continuous. 

The Radon-Nikodym property in locally convex spaces has been investigated by many authors; see, e.g., \cite{blondia,chi,gili,kurana,sambucini} and the doctoral dissertation of Y. J. Yaseen \cite{yasen}. 
We say that $X$ has the Lipschitz-Radon-Nikodym property if any Lipschitz continuous function $f : [0,+\infty) \rightarrow X$ is differentiable almost
everywhere and its first derivative is locally integrable (in the sense of \cite[Definition 1.1.2(i)]{FKP}). 

In the setting of complex Banach spaces, the Lipschitz-Radon-Nikodym property is equivalent with the Radon-Nikodym property and the Gel'fand property; it is completely without scope of this paper to further relate these concepts in the setting of locally convex spaces (cf. \cite[Section 1.2]{a43}, \cite{ul}, \cite[p. 371]{grose} and \cite[Definition 3.1, p. 36]{yasen} for more details). We will only emphasize here that a Fr\' echet space $X$ has the Lipschitz-Radon-Nikodym property if any Lipschitz continuous function $f : [0,+\infty) \rightarrow X$ is differentiable almost
everywhere because its first derivative is automatically locally integrable (the assumption $p(f(t)-f(t'))\leq M_{p}|t-t'|$ for all $t,\ t'\in [0,T]$ implies $p(f^{\prime}(t))\leq M_{p}$ for a.e. $t\in [0,T]$ and this is sufficient to ensure the local integrability of $f^{\prime}(\cdot)$); to the best knowledge of the author, this cannot be so easily done in general SCLCSs. Since any Lipschitz continuous function $f : [0,+\infty) \rightarrow X$ is locally of bounded variation, we consequently obtain that any Fr\' echet space with the Gel'fand property has the  Lipschitz-Radon-Nikodym property (\cite{grose}).

\subsection{Widder-Arendt theorem in locally convex spaces}\label{oned}

In this subsection, we recall and provide some new results about the Widder-Arendt theorem in locally convex spaces. First of all,
we will slightly reformulate the integrated version of Widder-Arendt theorem in locally convex spaces (cf. \cite[Theorem 2.3]{x26};  part (iii)' is new and not considered elsewhere):

\begin{thm}\label{waxl}
Suppose that $\omega\in {\mathbb R}$, $X$ is an \emph{SCLCS}, $F : (\omega,+\infty) \rightarrow X$ and $(M_{p})_{p\in \circledast}$ is a family of positive real numbers.  
Then the following statements are equivalent:
\begin{itemize}
\item[(i)] $F(\cdot)$ is infinitely differentiable and for each seminorm $p\in \circledast$ we have
\begin{align}\label{mp-xl}
p\Bigl( F^{(v)}\bigl( \lambda\bigr) \Bigr) \leq M_{p}\frac{v!   }{\bigl(\lambda-\omega\bigr)^{v+1}},\quad \lambda>\omega,\ v \in {\mathbb N}_{0}.
\end{align}
\item[(i)'] There exists $a \geq \max(\omega,0)$ such that $F_{| (a,+\infty)}(\cdot)$ is infinitely differentiable and for each seminorm $p\in \circledast$ the estimate \eqref{mp-xl} holds for any $\lambda>a $ and $v \in {\mathbb N}_{0}.$
\item[(ii)] $F(\cdot )$ admits an analytical extension to the right half-plane $\{z  \in {\mathbb C} : \Re z>\omega\} ,$ denoted by the same symbol, and for each seminorm $p\in \circledast$ we have
\begin{align*}
p\Bigl( F^{(v)}\bigl( \lambda\bigr) \Bigr) \leq M_{p}\frac{v!   }{\bigl(\Re \lambda-\omega\bigr)^{v+1}},\quad \Re \lambda>\omega,\ v \in {\mathbb N}_{0}.
\end{align*}
 \item[(iii)]
For every $r \in (0,1]$, there exists a continuous function $f_{r} : [0,+\infty) \rightarrow X$ such that $f_{r}(0)=0,$ 
\begin{align*} 
F\bigl(  \lambda \bigr)=\lambda^{r}\int^{+\infty}_{0}e^{-\lambda t }f_{r }(t)\, dt,\quad \lambda>\max(\omega,0),
\end{align*}
for each seminorm $p\in \circledast$ we have
\begin{align*} 
p\Bigl( f_{r}\bigl( t+h\bigr) -f_{r }\bigl( t\bigr)\Bigr) \leq \frac{2M_{p}}{r\Gamma (r)} \max\Bigl(e^{\omega (t+h)},1\Bigr)h^{r},\quad t\geq 0,\ h\geq 0,
\end{align*}
and the function $H_{r}(\cdot):=(g_{1-r} \ast_{0} f_{r})(\cdot)$ satisfies that for each seminorm $p\in \circledast$ we have
\begin{align*} 
p\Bigl(H_{r}\bigl( t+h\bigr) - H_{r}\bigl( t\bigr)\Bigr) \leq M_{p}e^{\omega t}\max\Bigl(e^{\omega h},1\Bigr)h,\quad t\geq 0,\ h\geq 0.
\end{align*}
\item[(iii)'] 
For every $r \in (0,1]$, there exists a continuous function $h_{r} : [0,+\infty) \rightarrow X$ such that $h_{r}(0)=0,$ 
\begin{align*}
F\bigl(  \lambda \bigr)=\bigl(\lambda-\omega \bigr)^{r}\int^{+\infty}_{0}e^{-(\lambda -\omega)t }h_{r }(t)\, dt,\quad \lambda>\omega ,
\end{align*}
for each seminorm $p\in \circledast$ we have
\begin{align*} 
p\Bigl( h_{r}\bigl( t+h\bigr) -h_{r }\bigl( t\bigr)\Bigr) \leq \frac{2M_{p}}{r\Gamma (r)}  h^{r},\quad t\geq 0,\ h\geq 0,
\end{align*} 
and the function $G_{r}(\cdot):=(g_{1-r} \ast_{0} h_{r})(\cdot)$ satisfies that for each seminorm $p\in \circledast$ we have
\begin{align}\label{hr2g}
p\Bigl(G_{r}\bigl( t+h\bigr) - G_{r}\bigl( t\bigr)\Bigr) \leq M_{p}  h,\quad t\geq 0,\ h\geq 0.
\end{align}
\end{itemize}
\end{thm}

\begin{proof} 
The implication (iii) $\Rightarrow$ (ii) follows from the proof of implication (ii) $\Rightarrow$ (i) in the above-mentioned theorem, with $a=\max(\omega,0)$, the implication (ii) $\Rightarrow$ (i)' is trivial and the implication (i) $\Rightarrow$ (iii) follows from the proof of implication (i)' $\Rightarrow$ (ii) in the above-mentioned theorem, with $a=\max(\omega,0)$. Therefore, the statements (i)', (ii) and (iii) are equivalent. The equivalence of these conditions with (i) follows from the trivial implications (ii) $\Rightarrow$ (i) $\Rightarrow$ (i)'. If (iii) holds, then we define the function $G:(0,+\infty) \rightarrow X$ by $G(\lambda):=F(\lambda+\omega),$ $\lambda>0.$ Then, due to (i), it follows that for each seminorm $p\in \circledast$ we have $
p( G^{(v)}( \lambda) ) \leq M_{p}v!  /\lambda^{v+1},$
$\lambda>0 $, $v \in {\mathbb N}_{0}.$ Applying (i) $\Rightarrow$ (iii) to function $G(\cdot)$, with $\omega=0$, we get (iii)'. Finally, if (iii)' holds, then (iii) holds for function $G(\cdot)$ defined above, with $\omega=0.$ Hence, (i) holds for function $G(\cdot)$, with $\omega=0$, so that for each seminorm $p\in \circledast$ we have $
p( G^{(v)}( \lambda) ) \leq M_{p}v!  /\lambda^{v+1},$
$\lambda>0 $, $v \in {\mathbb N}_{0}.$ This simply implies (i) and completes the proof.
\end{proof}

The following result is not well formulated in the setting of locally convex spaces by now (cf. \cite[Theorem 2.4.2]{a43} for the Banach space setting):

\begin{thm}\label{waxl1}
Suppose that $\omega\in {\mathbb R}$, $X$ is an \emph{SCLCS}, $X$ has the Lipschitz-Radon-Nikodym property, $F : (\omega,+\infty) \rightarrow X$ and $(M_{p})_{p\in \circledast}$ is a family of positive real numbers.  
Then the following statements are equivalent:
\begin{itemize}
\item[(i)] $F(\cdot)$ is infinitely differentiable and \eqref{mp-xl} holds for all $p\in \circledast$, $\lambda>\omega $ and $v \in {\mathbb N}_{0}.$ 
\item[(ii)] There exists $f \in L_{loc}^{1}([0,+\infty): X)$ such that 
\begin{align*}
F\bigl(  \lambda \bigr)=\int^{+\infty}_{0}e^{-\lambda t }f(t)\, dt,\quad \lambda>\omega,
\end{align*}
and
for each seminorm $p\in \circledast$ we have $p( f( t)) \leq M_{p}e^{\omega t},$ $t\geq 0$.   
\end{itemize}
\end{thm}

\begin{proof}
Let us assume first (ii). Then $F(\cdot)$ is analytic in the right half-plane $\{ \lambda \in {\mathbb C} : \Re \lambda> \omega\}$ and
\begin{align*}
\frac{d^{v}}{d\lambda^{v}}
F\bigl(  \lambda \bigr)=\int^{+\infty}_{0}e^{-\lambda t }(-t)^{v}f(t)\, dt,\quad \lambda>\omega;
\end{align*} 
cf. the equation \cite[(34), p. 58]{FKP}. Then
a simple calculation yields that (i) holds. To prove the converse statement, let us consider the function $h_{1}(\cdot)$ given in part (iii)' of Theorem \ref{waxl}. Then we know that \eqref{hr2g} holds with $r=1$ as well as that for each seminorm $p\in \circledast$ the estimate \eqref{hr2g} holds with the function $G_{r}(\cdot)$ replaced by the function $f_{1}(\cdot)$ therein. Therefore, for each seminorm $p\in \circledast$ we have $p(h_{1}( t+h) - h_{1}( t)) \leq M_{p} h,$
$t\geq 0,$ $h\geq 0$ and the function $h_{1}(\cdot)$ is Lipschitz continuous. Since $X$ has the Lipschitz-Radon-Nikodym property, its first derivative $h_{1}^{\prime}(t)$ exists for a.e. $t\geq 0$ and it is locally integrable on the non-negative real axis. It is clear that for each seminorm $p\in \circledast$ we have $p(h_{1}^{\prime}( t))\leq M_{p}$ for a.e. $t\geq 0.$ Applying the formula for partial integration established in \cite[Theorem 1.1.4(iii)]{FKP} and keeping in mind the equality $h_{1}(0)=0,$ it readily follows that 
$$
F(\lambda)=\int^{+\infty}_{0}e^{-(\lambda-\omega)t}h_{1}^{\prime}( t)\, dt,\quad \lambda >\omega.
$$
This yields 

$$
F^{(v)}(\lambda)=\int^{+\infty}_{0}e^{-(\lambda-\omega)t}(-t)^{v}h_{1}^{\prime}( t)\, dt,\quad \lambda >\omega,\ v\in {\mathbb N}_{0},
$$
which immediately implies (i).
\end{proof}

\subsection{Multidimensional vector-valued Laplace transform}\label{sdf}

The double Laplace transform of scalar-valued functions was initially analyzed  by D. L. Bernstein \cite{bern0} (1939), J. C. Jaeger \cite{jager} (1939--1941) and L. Amerio \cite{amerio} (1940). In our recent paper \cite{mvlt}, we have analyzed the multidimensional Laplace transform of functions with values in SCLCSs.
We recall the following notion:

\begin{defn}\label{svi}
Suppose that $X$ is an SCLCS, $f: [0,+\infty)^{n}\rightarrow X$ is a locally integrable function and $(\lambda_{1},...,\lambda_{n})\in {\mathbb C}^{n}$. If 
\begin{align*} 
F\bigl(\lambda_{1},...,\lambda_{n}\bigr)&:=\lim_{t_{1}\rightarrow +\infty;...;t_{n}\rightarrow +\infty}\int^{t_{1}}_{0}...\int^{t_{n}}_{0}e^{-\lambda_{1}s_{1}-...-\lambda_{n}s_{n}}f\bigl(s_{1},...,s_{n}\bigr)\, ds_{1}\, ...\, ds_{n}
\\& :=\int^{+\infty}_{0}...\int^{+\infty}_{0}e^{-\lambda_{1}t_{1}-...-\lambda_{n}t_{n}}f\bigl(t_{1},...,t_{n}\bigr)\, dt_{1}\, ...\, dt_{n},
\end{align*}
exists in topology of $X,$ 
then we say that the Laplace integral $({\mathcal Lf})(\lambda_{1},...,\lambda_{n}):=\hat{f}(\lambda_{1},...,\lambda_{n}):=F(\lambda_{1},...,\lambda_{n})$ exists.
We define the region of convergence of Laplace integral $
\Omega(f)$ by
$
\Omega(f):=\{ (\lambda_{1},...,\lambda_{n})\in {\mathbb C}^{n} : F(\lambda_{1},...,\lambda_{n})\mbox{ exists}\}.
$
Furthermore, if
for each seminorm $p\in \circledast$ we have
\begin{align*}
\int^{+\infty}_{0}...\int^{+\infty}_{0}p\Bigl(e^{-\lambda_{1}t_{1}-...-\lambda_{n}t_{n}}f\bigl(t_{1},...,t_{n}\bigr)\Bigr)\, dt_{1}\, ...\, dt_{n}<+\infty,
\end{align*}
then we say that the Laplace integral $F(\lambda_{1},...,\lambda_{n})$
converges absolutely. We define the region of absolute convergence of Laplace integral $
\Omega_{abs}(f)$ by
$
\Omega_{abs}(f):=\{ (\lambda_{1},...,\lambda_{n})\in {\mathbb C}^{n} : F(\lambda_{1},...,\lambda_{n})\mbox{ converges absolutely}\}.
$
Finally, we define $\Omega_{b}(f)$ as the set of all tuples $(\lambda_{1},...\lambda_{n})\in {\mathbb C}^{n}$ such that the set{\small
\begin{align*} 
\Biggl\{
\int^{t_{1}}_{0}\int^{t_{2}}_{0}...\int^{t_{n}}_{0}e^{-\lambda_{1}s_{1}-\lambda_{2}s_{2}-...-\lambda_{n}s_{n}}f\bigl( s_{1},s_{2},...,s_{n}\bigr)\, ds_{1} \, ds_{2}...\, ds_{n} :  t_{1}\geq 0,...,t_{n}\geq 0\Biggr\}
\end{align*}}
is bounded in $X.$
\end{defn}

If $f: [0,+\infty)^{n}\rightarrow X$ is locally integrable, $\emptyset \neq \Omega \subseteq \Omega(f) \cap \Omega_{b}(f)$ is open and $(\lambda_{1},...,\lambda_{n})\in \Omega,$ then the mapping $F: \Omega \rightarrow X$ is holomorphic and we have
\begin{align}\notag &
F^{(v_{1},...,v_{n})}(\lambda_{1},...,\lambda_{n})=(-1)^{v_{1}+...+v_{n}} \\\label{izvodi} & 
\times
 \Biggl[ \int^{+\infty}_{0}...\int^{+\infty}_{0}e^{-\lambda_{1}t_{1}-...-\lambda_{n}t_{n}}t_{1}^{v_{1} }\cdot ...\cdot  t_{n}^{v_{n}} f\bigl(t_{1},...,t_{n}\bigr)\, dt_{1}\, ...\, dt_{n} \Biggr],
\end{align}
for all $\lambda =(\lambda_{1},...,\lambda_{n})\in \Omega$ and $(v_{1},...,v_{n})\in {\mathbb N}_{0}^{n};$
for the basic information concerning holomorphic vector-valued functions of several variables, we refer the reader to the research article \cite{kruse} by K. Kruse and list of references quoted in \cite{mvlt}.  

\section{Multidimensional Widder-Arendt theorem in locally convex spaces with Lipschitz-Radon-Nikodym property}\label{lrnp}

The main result of this section reads as follows:

\begin{thm}\label{wa2}
Let $\omega_{1}\in {\mathbb R},\ ...,\ \omega_{n}\in {\mathbb R}$, $X$ is an \emph{SCLCS}, $X$ has the Lipschitz-Radon-Nikodym property, $(M_{p})_{p\in \circledast}$ is a family of positive real numbers and $F : (\omega_{1},+\infty) \times (\omega_{2},+\infty)\times \cdots \times (\omega_{n},+\infty) \rightarrow X.$
Let us consider the following statements:
\begin{itemize}
\item[(i)] $F(\cdot , ..., \cdot)$ is infinitely differentiable and for each seminorm $p\in \circledast$ we have
\begin{align}\label{mprp}
p\Bigl( F^{(v_{1},...,v_{n})}\bigl( \lambda_{1},...,\lambda_{n}\bigr) \Bigr) \leq M_{p}\frac{v_{1}! \cdots v_{n}!}{\bigl(\lambda_{1}-\omega_{1}\bigr)^{v_{1}+1}\cdots \bigl(\lambda_{n}-\omega_{n}\bigr)^{v_{n}+1}},\quad 
\end{align}
for any $\lambda_{1}>\omega_{1}, ...,$ $\lambda_{n}>\omega_{n} $ and $(v_{1},...,v_{n}) \in {\mathbb N}_{0}^{n}.$
\item[(ii)] There exists $f \in L_{loc}^{1}([0,+\infty)^{n}: X)$ such that 
\begin{align*}
F\bigl(  \lambda \bigr)=\int^{+\infty}_{0}...\int^{+\infty}_{0}e^{-\lambda_{1}t_{1}-...-\lambda_{n}t_{n}}f\bigl(t_{1},...,t_{n}\bigr)\, dt_{1}\, ...\, dt_{n},\quad \lambda_{i}>\omega_{i}\ \ (1\leq i\leq n),
\end{align*}
and
for each seminorm $p\in \circledast$ we have $p( f( t_{1},...,t_{n})) \leq M_{p}e^{\omega_{1} t_{1}+...+\omega_{n}t_{n}},$ $t=(t_{1},...,t_{n})\in [0,+\infty)^{n}$.   
\end{itemize}
Then we have \emph{(ii)} $\Rightarrow$ \emph{(i)}. Moreover, if $X$ is a Fr\' echet space, then \emph{(i)} $\Leftrightarrow$ \emph{(ii)}.
\end{thm}

\begin{proof}
If (ii) holds, then 
the mapping $(\lambda_{1},...,\lambda_{n}) \mapsto F(\lambda_{1},...,\lambda_{n}),$ $\Re \lambda_{1}>\omega_{1},...,$ $ \Re \lambda_{n}>\omega_{n}$ is analytic and \eqref{izvodi} holds 
for $(v_{1},...,v_{n}) \in {\mathbb N}_{0}$ and $\Re \lambda_{1}>\omega_{1},...,\ \Re \lambda_{n}>\omega_{n}.$
Since 
for each seminorm $p\in \circledast$ we have $p( f( t_{1},...,t_{n})) \leq M_{p}e^{\omega_{1} t_{1}+...+\omega_{n}t_{n}},$ $t=(t_{1},...,t_{n})\in [0,+\infty)^{n}$, (i) follows from a simple computation involving the Fubini theorem. 

Assume now that (i) holds and $X$ is a Fr\' echet space. We will prove that (ii) holds by induction on $n.$ By Theorem \ref{waxl1}, (ii) is true in dimension $n=1;$ suppose that (ii) is true in any dimension $<n,$ with the meaning clear. Let us fix $\lambda_{2}>\omega_{2},...,\ \lambda_{n}>\omega_{n}$ and let us consider the function 
$F_{\lambda_{2},...,\lambda_{n}} : (\omega_{1},+\infty) \rightarrow X$ defined by $F_{\lambda_{2},...,\lambda_{n}}(\lambda_{1}):=F(\lambda_{1},\lambda_{2},...,\lambda_{n}),$ $\lambda_{1}>\omega_{1}.$
Since for each seminorm $p\in \circledast$ we have 
\begin{align*} &
p\Bigl( F^{(v_{1})}\bigl( \lambda_{1}\bigr)\Bigr)=p\Bigl( F^{(v_{1},0,...,0)} \bigl(\lambda_{1},\lambda_{2},...,\lambda_{n}\bigr) \Bigr)
\\& \leq M_{p}\frac{v_{1}!}{\bigl(\lambda_{1}-\omega_{1}\bigr)^{v_{1}+1}\bigl(\lambda_{2}-\omega_{2}\bigr)\cdots \bigl(\lambda_{n}-\omega_{n}\bigr)},\quad v_{1}\in {\mathbb N}_{0},\ \lambda_{1}>\omega_{1},
\end{align*}
an application of Theorem \ref{waxl1} yields the existence of a locally integrable function $f_{\lambda_{2},...,\lambda_{n}} : [0,+\infty) \rightarrow X$ such that
\begin{align}\label{123}
F\bigl(\lambda_{1},\lambda_{2},...,\lambda_{n}\bigr)=\int^{+\infty}_{0}e^{-\lambda_{1}t_{1}}f_{\lambda_{2},...,\lambda_{n}}\bigl( t_{1}\bigr)\, dt_{1},\quad \lambda_{1}>\omega_{1},
\end{align}
and for each seminorm $p\in \circledast$ we have
\begin{align}\label{idijot}
p\Bigl( f_{\lambda_{2},...,\lambda_{n}}\bigl( t_{1}\bigr)\Bigr) \leq M_{p}e^{\omega_{1}t_{1}}\bigl(\lambda_{2}-\omega_{2}\bigr)^{-1}\cdot ... \cdot \bigl(\lambda_{n}-\omega_{n}\bigr)^{-1},\quad t_{1}\geq 0.
\end{align}
Since $X$ is a Fr\' echet space, the Post-Widder inversion formula and \eqref{123}-\eqref{idijot} together imply that there exists a Lebesgue measurable set $N\subseteq (0,\infty)$ whose Lebesgue measure is equal to zero, such that
\begin{align*}
f_{\lambda_{2},...,\lambda_{n}}\bigl( t_{1}\bigr)=\lim_{k\rightarrow +\infty}\frac{(-1)^{k}}{k!}\Biggl( \frac{k}{t_{1}}\Biggr)^{k+1}F^{(k,0,...,0)}\Biggl(\frac{k}{t_{1}},\lambda_{2},...,\lambda_{n}\Biggr),\ t_{1} \in (0,+\infty) \setminus N.
\end{align*} 
Now we will prove that, for every $t_{1}\in (0,+\infty) \setminus N,$ the mapping $(\lambda_{2},...,\lambda_{n}) \rightarrow f_{\lambda_{2},...,\lambda_{n}} ( t_{1} ),$ $\lambda_{2}>\omega_{2},...,\ \lambda_{n}>\omega_{n}$ is infinitely differentiable as well as that 
\begin{align}\label{vb} 
\frac{\partial^{v_{2}+...+v_{n}}}{\partial \lambda_{2}^{v_{2}}\cdot ... \cdot  \partial \lambda_{n}^{v_{n}}}
f_{\lambda_{2},...,\lambda_{n}}\bigl( t_{1}\bigr)=\lim_{k\rightarrow +\infty}\frac{(-1)^{k}}{k!}\Biggl( \frac{k}{t_{1}}\Biggr)^{k+1}F^{(k,v_{2},...,v_{n})}\Biggl(\frac{k}{t_{1}},\lambda_{2},...,\lambda_{n}\Biggr),
\end{align} 
for any $t_{1} \in (0,+\infty) \setminus N.$ Suppose first that $v_{2}=1$ and $v_{3}=...=v_{n}=0.$ Let $\delta>0$ be such that $\lambda_{2}-\delta>\omega_{2}$ and let $h\in (-\delta,\delta).$ If $p\in \circledast$, then we have{\scriptsize
\begin{align*} &
p\Biggl( \frac{f_{\lambda_{2}+h,...,\lambda_{n}}\bigl( t_{1}\bigr)-f_{\lambda_{2},...,\lambda_{n}}\bigl( t_{1}\bigr)}{h}-\lim_{k\rightarrow +\infty}\frac{(-1)^{k}}{k!}\Bigl( \frac{k}{t_{1}}\Bigr)^{k+1}F^{(k,1,0,...,0)}\bigl(  k/t_{1},\lambda_{2},...,\lambda_{n} \bigr)\Biggr)
\\& =\sup_{x^{\ast} \in U_{p}^{\circ}} \Biggl| \Biggl \langle x^{\ast} ,\frac{f_{\lambda_{2}+h,...,\lambda_{n}}\bigl( t_{1}\bigr)-f_{\lambda_{2},...,\lambda_{n}}\bigl( t_{1}\bigr)}{h}-\lim_{k\rightarrow +\infty}\frac{(-1)^{k}}{k!}\Bigl( \frac{k}{t_{1}}\Bigr)^{k+1}F^{(k,1,0,...,0)}\bigl( \frac{k}{t_{1}},\lambda_{2},...,\lambda_{n} \bigr)\Biggr \rangle \Biggr|
\\& \leq \sup_{x^{\ast} \in U_{p}^{\circ}}\limsup_{k\rightarrow +\infty}\frac{1}{k!}\Bigl( \frac{k}{t_{1}}\Bigr)^{k+1}
\\& \times \Biggl| \Biggl \langle x^{\ast} ,   \frac{F^{(k,0,0,...,0)}\bigl( \frac{k}{t_{1}},\lambda_{2}+h,...,\lambda_{n} \bigr)-F^{(k,0,0,...,0)}\bigl( \frac{k}{t_{1}},\lambda_{2},...,\lambda_{n} \bigr)-hF^{(k,1,0,...,0)}\bigl( \frac{k}{t_{1}},\lambda_{2},...,\lambda_{n} \bigr)}{h}      \Biggr \rangle \Biggr|
\end{align*}
\begin{align*}
\\ & \leq h\sup_{x^{\ast} \in U_{p}^{\circ}}\limsup_{k\rightarrow +\infty}\Biggl[\frac{1}{k!}\Bigl( \frac{k}{t_{1}}\Bigr)^{k+1} \sup_{\lambda' \in [\lambda_{2}-\delta,\lambda_{2}+\delta]}\Biggl| \Biggl \langle x^{\ast} , F^{(k,2,0,...,0)}\bigl( k/t_{1},\lambda_{2},...,\lambda_{n} \bigr) \Biggr \rangle \Biggr| \Biggr]
\\& \leq h\limsup_{k\rightarrow +\infty}\Biggl[\frac{1}{k!}\Bigl( \frac{k}{t_{1}}\Bigr)^{k+1} \sup_{\lambda' \in [\lambda_{2}-\delta,\lambda_{2}+\delta]}p\Biggl( F^{(k,2,0,...,0)}\bigl( k/t_{1},\lambda_{2},...,\lambda_{n} \bigr)\Biggr)\Biggr]
\\& \leq  h\bigl( \lambda_{2}-\delta-\omega_{2})^{-3}\bigl(\lambda_{3}-\omega_{3}\bigr)^{-1}\cdot ... \cdot \bigl(\lambda_{n}-\omega_{n}\bigr)^{-1}\limsup_{k\rightarrow +\infty}\Biggl[\frac{1}{k!}\Bigl( \frac{k}{t_{1}}\Bigr)^{k+1}\frac{2M_{p}k!}{((k/t_{1})-\omega_{1})^{k+1}}\Biggr]
\\& =2M_{p}h\bigl( \lambda_{2}-\delta-\omega_{2})^{-3}\bigl(\lambda_{3}-\omega_{3}\bigr)^{-1}\cdot ... \cdot \bigl(\lambda_{n}-\omega_{n}\bigr)^{-1}e^{\omega_{1}t_{1}} \rightarrow 0,
\end{align*}}
as $h\rightarrow 0.$ Repeating this argumentation, we obtain that \eqref{vb} holds
for $\lambda_{2}>\omega_{2},...,\ \lambda_{n}>\omega_{n}$ and $(v_{2},...,v_{n})\in {\mathbb N}_{0}^{n-1}.$ 

Further on, by the induction hypothesis, for every $t_{1}\in (0,+\infty) \setminus N,$ there exists a locally integrable function $f_{t_{1}} : [0,+\infty)^{n-1} \rightarrow X$ such that{\small
\begin{align}\label{du}
f_{\lambda_{2},...,\lambda_{n}}\bigl( t_{1}\bigr) =\int^{+\infty}_{0}...\int^{+\infty}_{0}e^{-\lambda_{2}t_{2}-...-\lambda_{n}t_{n}}f_{t_{1}} \bigl(t_{2},...,t_{n}\bigr)\, dt_{2}\, ...\, dt_{n},\ \lambda_{i}>\omega_{i}\ \ (2\leq i\leq n),
\end{align}}
and for each seminorm $p\in \circledast$ we have
\begin{align}\label{idijotq}
p\Bigl( f_{t_{1}} \bigl(t_{2},...,t_{n}\bigr)\Bigr) \leq M_{p}e^{\omega_{1}t_{1}+...+\omega_{n}t_{n}},\quad t_{1}\in (0,+\infty) \setminus N,\ \bigl(t_{2},...,t_{n}\bigr) \in [0,+\infty)^{n-1}.
\end{align}
Define now $f(t_{1},...,t_{n}):=f_{t_{1}} (t_{2},...,t_{n})$ for a.e. $(t_{1},...,t_{n})\in [0,+\infty)^{n}$; then we have $f\in L_{loc}^{1}([0,+\infty)^{n} : X).$ Using the Fubini theorem, \eqref{123} and \eqref{idijotq}, we get 
\begin{align*}&
F\bigl(\lambda_{1},\lambda_{2},...,\lambda_{n}\bigr)=\int^{+\infty}_{0}e^{-\lambda_{1}t_{1}}f_{\lambda_{2},...,\lambda_{n}}\bigl( t_{1}\bigr)\, dt_{1}
\\ & = \int^{+\infty}_{0}e^{-\lambda_{1}t_{1}}\int^{+\infty}_{0}...\int^{+\infty}_{0}e^{-\lambda_{2}t_{2}-...-\lambda_{n}t_{n}}f_{t_{1}} \bigl(t_{2},...,t_{n}\bigr)\, dt_{2}\, ...\, dt_{n} \, dt_{1}
\\& = \int^{+\infty}_{0}...\int^{+\infty}_{0}e^{-\lambda_{1}t_{1}-\lambda_{2}t_{2}-...-\lambda_{n}t_{n}}f \bigl(t_{1},t_{2},...,t_{n}\bigr)\, dt_{1}\, ...\, dt_{n},
\end{align*}
for $\lambda_{1}>\omega_{1},...,\ \lambda_{n}>\omega_{n}.$ Keeping in mind the estimate \eqref{idijotq}, this implies that (ii) is true in dimension $n,$ which completes the proof of theorem. 
\end{proof}

\section{Integrated form of multidimensional Widder--Arendt theorem in arbitrary locally convex spaces}\label{kon}

In this section, we will prove the following integrated form of Widder--Arendt theorem in arbitrary locally convex spaces:

\begin{thm}\label{maha}
Suppose that $\omega_{1}\in {\mathbb R},...,$ $\omega_{n}\in {\mathbb R}$, $F : (\omega_{1},+\infty) \times ... \times (\omega_{n},+\infty) \rightarrow X$ and $(M_{p})_{p\in \circledast}$ is a family of positive real numbers.  
Then the following statements are equivalent:
\begin{itemize}
\item[(i)] $F(\cdot)$ is infinitely differentiable and the estimate \eqref{mprp} holds for any $p\in \circledast$, $\lambda_{1}>\omega_{1},..., $ $\lambda_{n}>\omega_{n} $ and $(v_{1},..., v_{n}) \in {\mathbb N}_{0}^{n}.$
\item[(i)'] There exist real numbers  $a_{1}\geq \max(\omega_{1},0),...,$ $a_{n}\geq \max(\omega_{n},0) $ such that  $F_{| (a_{1},+\infty)\times ... \times  (a_{n},+\infty)}(\cdot)$ is infinitely differentiable and for each seminorm $p\in \circledast$ the estimate \eqref{mprp} holds
for any $\lambda_{1}>a_{1},..., $ $\lambda_{n}>a_{n} $ and $(v_{1},..., v_{n}) \in {\mathbb N}_{0}^{n}.$
\item[(ii)] $F(\cdot )$ admits an analytical extension to the region $\{z_{1}  \in {\mathbb C} : \Re z_{1}>\omega_{1}\} \times ... \times \{z_{n}  \in {\mathbb C} : \Re z_{n}>\omega_{n}\},$ denoted by the same symbol, and for each seminorm $p\in \circledast$ we have
\begin{align*}
p\Bigl( F^{(v_{1},...,v_{n})}\bigl( \lambda_{1}, ..., \lambda_{n}\bigr) \Bigr) \leq M_{p}\frac{v_{1}! \cdot ... \cdot   v_{n}!}{\bigl(\Re \lambda_{1}-\omega_{1}\bigr)^{v_{1}+1}\cdot ... \cdot \bigl(\Re \lambda_{n}-\omega_{n}\bigr)^{v_{n}+1}}, 
\end{align*}
for any $\Re \lambda_{1}>\omega_{1},..., $ $\Re \lambda_{n}>\omega_{n} $ and $(v_{1},..., v_{n}) \in {\mathbb N}_{0}^{n}.$
\item[(iii)] We have \emph{(iii.1)-(iii.2)}, where:
\begin{itemize}
\item[(iii.1)]
For every $r_{1} \in (0,1],...$, $r_{n} \in (0,1],$ there exists a continuous function $f_{r_{1},...,r_{n}} : [0,+\infty) \rightarrow X$ such that $f_{r_{1},...,r_{n}}(0)=0$, for each seminorm $p\in \circledast$ we have
\begin{align}\notag &
p\Bigl(f_{r_{1},...,r_{n}}\bigl(t_{1},...,t_{n}\bigr)\Bigr)
\\ \label{prc0}& \leq M_{p}\int^{t_{1}}_{0}\cdots \int^{t_{n}}_{0}g_{r_{1}}\bigl(t_{1}-s_{1}\bigr) \cdot ...\cdot g_{r_{n}}\bigl(t_{n}-s_{n}\bigr)e^{\omega_{1}s_{1}+...+\omega_{n}s_{n}}\, ds_{1} ...\, ds_{n},\ 
\end{align}
for any $\bigl(t_{1},...,t_{n}\bigr) \in [0,+\infty)^{n},$ and
\begin{align}\label{prc1}
F\bigl( \lambda_{1},...,\lambda_{n}\bigr)=\lambda_{1}^{r_{1}}\cdot ... \cdot \lambda_{n}^{r_{n}}\int^{+\infty}_{0}...\int^{+\infty}_{0}e^{-\lambda_{1}t_{1}-...-\lambda_{n}t_{n}}f_{r_{1},...,r_{n}}\bigl(t_{1},...,t_{n}\bigr)\, dt_{1}\, ...\, dt_{n},
\end{align} 
for $\lambda_{1}>\max\bigl(\omega_{1},0\bigr),..., $ $\lambda_{n}>\max\bigl(\omega_{n},0\bigr)$.
\item[(iii.2)] For each seminorm $p\in \circledast$ and for each functional $x^{\ast} \in U_{p}^{\circ}$, there exists a locally integrable function $f_{x^{\ast}} : [0,+\infty)^{n} \rightarrow X$ such that $p(f_{x^{\ast}}(t_{1},...,t_{n})) \leq M_{p}\exp(\omega_{1}t_{1}+...+\omega_{n}t_{n})$, $(t_{1},...,t_{n}) \in [0,+\infty)^{n}$ and
\begin{align}\label{nl}
\Bigl \langle x^{\ast},  F\bigl( \lambda_{1},...,\lambda_{n}\bigr)\Bigr \rangle =\int^{+\infty}_{0}...\int^{+\infty}_{0}e^{-\lambda_{1}t_{1}-...-\lambda_{n}t_{n}}\Bigl \langle x^{\ast},  f_{x^{\ast}}\bigl(t_{1},...,t_{n}\bigr) \Bigr \rangle \, dt_{1}\, ...\, dt_{n},
\end{align} 
for any $\lambda_{1}>\max\bigl(\omega_{1},0\bigr),..., $ $\lambda_{n}>\max\bigl(\omega_{n},0\bigr)$.
\end{itemize}
If this is the case, then we have{\small
\begin{align}\notag &
\Bigl \langle x^{\ast},  f_{r_{1},...,r_{n}}\bigl(t_{1},...,t_{n}\bigr)\Bigr \rangle 
\\\label{adf}& =\int^{t_{1}}_{0}\cdots \int^{t_{n}}_{0}g_{r_{1}}\bigl(t_{1}-s_{1}\bigr) \cdot ... \cdot g_{r_{n}}\bigl(t_{n}-s_{n}\bigr)\Bigl \langle x^{\ast},f_{x^{\ast}}\bigl(s_{1},...,s_{n}\bigr) \Bigr \rangle \, ds_{1} ... \, ds_{n},
\end{align}}
for any $x^{\ast} \in U_{p}^{\circ}$ and $(t_{1},...,t_{n})\in [0,+\infty)^{n}.$
\item[(iii)']
We have \emph{(iii.1)'-(iii.2)'}, where:
\begin{itemize}
\item[(iii.1)']
For every $r_{1} \in (0,1],...$, $r_{n} \in (0,1],$ there exists a continuous function $h_{r_{1},...,r_{n}} : [0,+\infty)^{n} \rightarrow X$ such that $h_{r_{1},...,r_{n}}(0)=0$, for each seminorm $p\in \circledast$ we have
\begin{align}\label{prc0w}
p\Bigl(h_{r_{1},...,r_{n}}\bigl(t_{1},...,t_{n}\bigr)\Bigr)
 \leq M_{p}g_{r_{1}+1}\bigl(t_{1}\bigr) \cdot ...\cdot g_{r_{n}+1}\bigl(t_{n}\bigr),\quad \bigl(t_{1},...,t_{n}\bigr) \in [0,+\infty)^{n},
\end{align}
and
\begin{align}\notag &
F\bigl( \lambda_{1},...,\lambda_{n}\bigr)=\bigl(\lambda_{1}-\omega_{1}\bigr)^{r_{1}}\cdot ... \cdot \bigl(\lambda_{n}-\omega_{n}\bigr)^{r_{n}}
\\& \label{prc1w} \times \int^{+\infty}_{0}...\int^{+\infty}_{0}e^{-(\lambda_{1}-\omega_{1})t_{1}-...-(\lambda_{n}-\omega_{n})t_{n}}h_{r_{1},...,r_{n}}\bigl(t_{1},...,t_{n}\bigr)\, dt_{1}\, ...\, dt_{n},
\end{align} 
for any $\lambda_{1}> \omega_{1} ,..., $ $\lambda_{n}> \omega_{n} $.
\item[(iii.2)'] For each seminorm $p\in \circledast$ and for each functional $x^{\ast} \in U_{p}^{\circ}$, there exists a locally integrable function $h_{x^{\ast}} : [0,+\infty)^{n} \rightarrow X$ such that $p(h_{x^{\ast}}(t_{1},...,t_{n})) \leq M_{p}$, $(t_{1},...,t_{n}) \in [0,+\infty)^{n}$ and{\small
\begin{align*}
\Bigl \langle x^{\ast},  F\bigl( \lambda_{1},...,\lambda_{n}\bigr)\Bigr \rangle =\int^{+\infty}_{0}...\int^{+\infty}_{0}e^{-(\lambda_{1}-\omega_{1})t_{1}-...-(\lambda_{n}-\omega_{n})t_{n}}\Bigl \langle x^{\ast},  h_{x^{\ast}}\bigl(t_{1},...,t_{n}\bigr) \Bigr \rangle \, dt_{1}\, ...\, dt_{n},
\end{align*}}
for any $\lambda_{1}>\omega_{1},..., $ $\lambda_{n}>\omega_{n}$.
\end{itemize}
\end{itemize}
\end{thm}

\begin{proof}
It is clear that (ii) implies (i) and (i)'. If (i)' holds, then we can use the usual series expansion to prove (ii); see also the proof of \cite[Theorem 2.1, p. 78]{li-shaw}. Therefore, the assertions (i), (i)' and (ii) are equivalent. 

If (iii) holds, then we can use the convolution theorem for Laplace transform and the uniqueness theorem for Laplace transform in order to see that \eqref{adf} holds for any $x^{\ast} \in U_{p}^{\circ}$ and $(t_{1},...,t_{n})\in [0,+\infty)^{n};$ cf. \cite{mvlt} for more details.
Now we will prove that (iii) implies (ii). Let $r_{1}=...=r_{n}=1 $ and $p\in \circledast$. Using \eqref{prc0}-\eqref{prc1}, it readily follows that $F(\cdot )$ admits an analytical extension to the region $\{z_{1}  \in {\mathbb C} : \Re z_{1}>\omega_{1}\} \times ... \times \{z_{n}  \in {\mathbb C} : \Re z_{n}>\omega_{n}\}.$ Keeping in mind
this fact and the uniqueness theorem for analytic functions, it follows that \eqref{nl} holds for any $\Re \lambda_{1}> \omega_{1} ,..., $ $\Re \lambda_{n}> \omega_{n} $. Now,
by the generalization of norm formula given in \cite[Proposition 22.14]{meise}, we have{\small
\begin{align*} &
p\Bigl(  F^{(v_{1},...,v_{n})}\bigl( \lambda_{1}, ..., \lambda_{n}\bigr)\Bigr)=\sup_{x^{\ast}\in U_{p}^{\circ}}\Biggl|  \frac{\partial^{v_{1}+...+v_{n}}}{\partial \lambda_{1}^{v_{1}}\cdot ... \cdot \partial \lambda_{n}^{v_{n}}} \Bigl \langle x^{\ast},  F\bigl( \lambda_{1},...,\lambda_{n}\bigr)\Bigr \rangle  \Biggr|
\\& =  \sup_{x^{\ast}\in U_{p}^{\circ}}\Biggl| \int^{+\infty}_{0}...\int^{+\infty}_{0}e^{- \lambda_{1}t_{1}-...- \lambda_{n}t_{n}}\bigl( -t_{1}\bigr)^{v_{1}} \cdot ... \cdot  \bigl( -t_{n}\bigr)^{v_{n}} \Bigl \langle x^{\ast},  f_{x^{\ast}}\bigl(t_{1},...,t_{n}\bigr) \Bigr \rangle \, dt_{1}\, ...\, dt_{n} \Biggr|
\\& \leq \int^{+\infty}_{0}...\int^{+\infty}_{0}e^{-\Re \lambda_{1}t_{1}-...-\Re \lambda_{n}t_{n}}t_{1}^{v_{1}} \cdot ... \cdot  t_{n}^{v_{n}}p\Bigl( f_{x^{\ast}}\bigl(t_{1},...,t_{n}\bigr) \Bigr ) \, dt_{1}\, ...\, dt_{n} 
\\& \leq M_{p} \int^{+\infty}_{0}...\int^{+\infty}_{0}e^{-\Re \lambda_{1}t_{1}-...-\Re \lambda_{n}t_{n}}t_{1}^{v_{1}} \cdot ... \cdot  t_{n}^{v_{n}}e^{\omega_{1}t_{1}+...+\omega_{n}t_{n}} \, dt_{1}\, ...\, dt_{n} 
\\& =M_{p}\frac{v_{1}! \cdot ... \cdot   v_{n}!}{\bigl(\Re \lambda_{1}-\omega_{1}\bigr)^{v_{1}+1}\cdot ... \cdot \bigl(\Re \lambda_{n}-\omega_{n}\bigr)^{v_{n}+1}}, 
\end{align*}}
for any $\Re \lambda_{1}>\omega_{1},..., $ $\Re \lambda_{n}>\omega_{n} $ and $(v_{1},..., v_{n}) \in {\mathbb N}_{0}^{n}.$ This yields (ii).

The equivalence between (iii) and (iii)' can be shown using the same arguments as in the corresponding part of proof of Theorem \ref{waxl}. Therefore, it remains to be proved that (ii) implies (iii).  Since for each seminorm $p\in \circledast$ we have $|\langle x^{\ast},x \rangle| \leq p(x),$ $x\in X,$ $x^{\ast}\in U_{p}^{\circ},$ an application of Theorem \ref{wa2} implies that (iii.2) holds. 
 
The part
(iii.1) can be deduced similarly as in the one-dimensional case and we will only provide the main details of proof here. Let us consider the space $\tilde{X}\subseteq X^{\ast \ast}$ consisting of all linear functionals $\tilde{x}$ on $X^{\ast}$ such that 
$P_{G}(\tilde{x}):=\sup_{x^{\ast}\in G}|\langle \tilde{x},x^{\ast} \rangle|<+\infty$ for each equicontinuous family $G\subseteq X^{\ast};$ the fundamental system of seminorms which defines a sequentially complete locally convex topology on $\tilde{X}$ is given by $(P_{G})_{G\in {\mathcal T}},$ where ${\mathcal T}$ denotes the collection of all equicontinuous families $G\subseteq X^{\ast}.$ We know that $X$ and $\tilde{X}$ are linearly and topologically isomorphic (\cite{x26}).

Suppose now that the numbers $r_{1} \in (0,1],...$, $r_{n} \in (0,1]$ are given. Using Theorem \ref{wa2}, it follows that for each functional $x^{\ast}\in X^{\ast}$ there exists a function $f_{x^{\ast}} \in L_{loc}^{1}([0,+\infty)^{n})$ such that{\small
\begin{align*}
\Bigl \langle x^{\ast} , F\bigl(  \lambda \bigr) \Bigr \rangle=\int^{+\infty}_{0}...\int^{+\infty}_{0}e^{-\lambda_{1}t_{1}-...-\lambda_{n}t_{n}}f_{x^{\ast}}\bigl(t_{1},...,t_{n}\bigr)\, dt_{1}\, ...\, dt_{n},\ \lambda_{i}>\omega_{i}\ \ (1\leq i\leq n).
\end{align*}}
Define{\small
$$
f_{r_{1},...,r_{n}; x^{\ast}}\bigl( t_{1},...,t_{n}\bigr):=\int^{t_{1}}_{0}\cdots \int^{t_{n}}_{0} g_{r_{1}}\bigl( t_{1}-s_{1}\bigr)\cdot ... \cdot g_{r_{n}}\bigl( t_{n}-s_{n}\bigr)f_{x^{\ast}}\bigl(s_{1},...,s_{n}\bigr)\, ds_{1}\, ...\, ds_{n},
$$}
for any $x^{\ast}\in X^{\ast}$ and $(t_{1},...,t_{n}) \in [0,+\infty)^{n}.$
Using the convolution theorem for multidimensional Laplace transform \cite[Proposition 2.17]{mvlt}, it follows that
{\small
\begin{align*}
\Bigl \langle x^{\ast} , F\bigl(  \lambda \bigr) \Bigr \rangle=\lambda_{1}^{r_{1}}\cdot ... \cdot \lambda_{n}^{r_{n}}\int^{+\infty}_{0}...\int^{+\infty}_{0}e^{-\lambda_{1}t_{1}-...-\lambda_{n}t_{n}}f_{r_{1},...,r_{n}; x^{\ast}}\bigl(t_{1},...,t_{n}\bigr)\, dt_{1}\, ...\, dt_{n},
\end{align*}}
for $\lambda_{1}>\max (\omega_{1},0),...,$ $  \lambda_{n}>\max (\omega_{n},0).$ For every $t=(t_{1},...,t_{n}) \in [0,+\infty)^{n},$ we define $F_{t} : X^{\ast} \rightarrow {\mathbb C}$ by $F_{t}(x^{\ast}):=\langle x^{\ast}, f_{r_{1},...,r_{n}; x^{\ast}}( t) \rangle ,$ $x^{\ast} \in X^{\ast}.$ Arguing as in the proof of \cite[Theorem 2.3]{x26}, we can show that $F_{t}\in \tilde{X}$ for all $t\in [0,+\infty)^{n}$ as well as that the mapping $t\mapsto F_{t},$ $t\in [0,+\infty)^{n}$ is continuous wth respect to the topology of $\tilde{X};$ let us only mention here that the linearity of mapping $F_{t}(\cdot)$ for fixed $t\in [0,+\infty)^{n}$ follows from the continuity of function $f_{r_{1},...,r_{n}; x^{\ast}}(\cdot)$ for each $x^{\ast} \in X^{\ast}$ and the uniqueness theorem for multidimensional Laplace transform \cite[Theorem 2.23]{mvlt}. The existence of a continuous function $f_{r_{1},...,r_{n}} : [0,+\infty) \rightarrow X$ such that $f_{r_{1},...,r_{n}}(0)=0$ and \eqref{prc1} holds
for $\lambda_{1}>\max(\omega_{1},0),..., $ $\lambda_{n}>\max(\omega_{n},0)$ now follows by identification of $X$ and $\tilde{X}$ through evaluation and the multidimensional Post-Widder formula \cite[Proposition 2.24]{mvlt}; cf. \cite[l. -1-l. -3, p. 457]{x26}.
Using the generalized norm formula, we can simply prove that
for each seminorm $p\in \circledast$ the estimate
\eqref{prc0} holds
for any $\bigl(t_{1},...,t_{n}\bigr) \in [0,+\infty)^{n}.$ 
\end{proof}

\begin{rem}\label{shaw-li}
Alternatively, the proof of implication (ii) $\Rightarrow$ (iii.1) can be given following the argumentation proposed by Y.-C. Li and S.-Y. Shaw 
in the proof of \cite[Theorem 2.1, pp. 79--80]{li-shaw}. In actual fact, we can use the complex characterization theorem for multidimensional vector-valued Laplace transform \cite[Theorem 2.13]{mvlt} in order to see that, for every real numbers $r_{1} \in (0,1],...$, $r_{n} \in (0,1],$ there exists a continuous function $w_{r_{1},...,r_{n}} : [0,+\infty)^{n} \rightarrow X$ such that $w_{r_{1},...,r_{n}}(0)=0$ and{\scriptsize
\begin{align*}  
F\bigl( \lambda_{1},...,\lambda_{n}\bigr)=\lambda_{1}^{r_{1}+1}\cdot ... \cdot \lambda_{n} ^{r_{n}+1}  
\int^{+\infty}_{0}...\int^{+\infty}_{0}e^{- \lambda_{1}t_{1}-...-\lambda_{n}t_{n}}w_{r_{1},...,r_{n}}\bigl(t_{1},...,t_{n}\bigr)\, dt_{1}\, ...\, dt_{n},
\end{align*}}
for $\Re \lambda_{1}>\max(\omega_{1},0),...,\ \Re \lambda_{n}>\max(\omega_{n},0).$ The function $f_{r_{1},...,r_{n}} (\cdot)$ constructed in (iii.1) and the function $w_{r_{1},...,r_{n}}(\cdot)$ satisfy the following equality
$$
w_{r_{1},...,r_{n}}\bigl(t_{1},...,t_{n}\bigr)=\int^{t_{1}}_{0}\cdots \int^{t_{n}}_{0}f_{r_{1},...,r_{n}}\bigl(s_{1},...,s_{n}\bigr)\, ds_{1}... \, ds_{n},\quad t_{1}\geq 0,\ ...,\ t_{n}\geq 0,
$$
so that the only problem in this approach is to prove that the partial derivative $(\partial^{n}/\partial t_{1} ... \partial t_{n})w_{r_{1},...,r_{n}}(\cdot)$ continuously exists on $[0,+\infty)^{n}.$ It seems very plausible that this problem can be solved with the help of estimate \eqref{mprp} and bipolar theorem.

In the proof of \cite[Theorem 2.1]{li-shaw}, the authors do not use the identification of pivot space $X$ and a closed subspace of its bidual $X^{\ast \ast}$, which has been originally used in the Banach space setting by M. Hieber in the proof of \cite[Theorem 3.2]{hiber}. The proof of \cite[Theorem 2.3]{x26}is also meaningful because 
T.-J. Xiao and J. Liang have explicitely constructed the space $\tilde{X}\subseteq X^{\ast \ast}$ which is linearly and topologically isomorphic to the initial space $X$ in the locally convex space setting.
\end{rem}

\end{document}